\documentclass[12pt]{amsart}
\usepackage{amssymb}

\newtheorem{theorem}{Theorem}[section]
\newtheorem{lemma}[theorem]{Lemma}
\newtheorem{proposition}[theorem]{Proposition}
\newtheorem{corollary}[theorem]{Corollary}
\newtheorem{remark}[theorem]{Remark}

\numberwithin{equation}{section}

\begin{document}

\baselineskip=15pt

\title[Affine Hermitian--Einstein connections]{Hermitian--Einstein
connections on principal bundles over flat affine manifolds}

\author[I. Biswas]{Indranil Biswas}

\address{School of Mathematics, Tata Institute of Fundamental
Research, Homi Bhabha Road, Bombay 400005, India}

\email{indranil@math.tifr.res.in}

\author[J. Loftin]{John Loftin}

\address{Department of Mathematics and Computer Science, Rutgers
University at Newark, Newark, NJ 07102, USA}

\email{loftin@rutgers.edu}



\date{}

\begin{abstract}
Let $M$ be a compact connected special flat affine manifold without
boundary equipped with a Gauduchon metric $g$ and a covariant constant
volume form. Let $G$ be either a connected reductive complex linear
algebraic group or the real locus of a split real form
of a complex reductive group. We
prove that a flat
principal $G$--bundle $E_G$ over $M$ admits a Hermitian--Einstein
structure if and only if $E_G$ is polystable. A polystable flat
principal $G$--bundle over $M$ admits a unique Hermitian--Einstein
connection. We also prove the existence and uniqueness of a
Harder--Narasimhan filtration for flat vector bundles over $M$.
We prove a Bogomolov type inequality for semistable vector
bundles under the assumption that the Gauduchon metric $g$ is
astheno--K\"ahler.
\end{abstract}

\maketitle

\section{Introduction}\label{sec1}

A flat affine manifold $M$ is a $C^\infty$ manifold equipped with a
flat torsionfree connection on $TM$. Equivalently, a flat affine
structure on a manifold is provided by an atlas of coordinate charts
whose transition functions are all affine maps $x\mapsto Ax+b$. The
total space of the tangent bundle $TM$ of a flat affine manifold
admits a complex structure, with the transition maps being $z\mapsto
Az+b$, for $z=x+\sqrt{-1}y$, with $y$ representing the fiber
coordinates. There is a dictionary between holomorphic objects on
$TM$ which are invariant in the fiber directions  and locally
constant objects on $M$ (cf. \cite{Lo}). This correspondence between
affine and complex manifolds has recently become prominent as a part
of the mirror conjecture of Strominger--Yau--Zaslow (in this case,
each fiber of the tangent bundle $TM\,\longrightarrow\,M$ is
quotiented by a lattice to form a special Lagrangian torus in a
Calabi--Yau manifold). In particular, a flat vector bundle over $M$
naturally extends to a holomorphic vector bundle over $TM$.

We briefly recall the set--up and the main result of \cite{Lo}. An
affine manifold $M$ is called special if the induced flat connection
on the line bundle $\bigwedge^{\rm top}TM$ has trivial holonomy. Let
$M$ be a compact connected special flat affine manifold without
boundary. Fix a nonzero flat section $\nu$ of $\bigwedge^{\rm
top}TM$ (under our dictionary, $\nu$ corresponds to a holomorphic
volume form on the total space of  $TM$). Also fix an affine
Gauduchon metric $g$ on $M$. This allows us to define the degree of
a flat vector bundle over $M$ (see Section \ref{sec2} below); we
will consider both real and complex vector bundles. By a flat vector
bundle we will always mean a vector bundle equipped with a flat
connection (and thus locally constant transition functions). For a
flat vector bundle $V\, \longrightarrow\, M$, its degree is denoted
by ${\rm deg}_g(V)$, and its slope ${\rm deg}_g(V)/{\rm rank}(V)$ is
denoted by $\mu_g(V)$.

Once degree is defined, we can define semistable, stable and polystable
flat vector bundles over $M$ by imitating the corresponding definitions
for holomorphic vector bundles over compact Gauduchon manifolds.
Similarly, Hermitian--Einstein metrics on flat vector bundles
over $M$ are defined by imitating the definition of
Hermitian--Einstein metrics on holomorphic vector bundles
over Gauduchon manifolds.

The following theorem is proved in \cite{Lo}, with the main technical
part being to use estimates to prove a version of the
Donaldson--Uhlenbeck--Yau Theorem on existence
of Hermitian--Einstein connections on stable vector bundles.

\begin{theorem}[\cite{Lo}]\label{thm0}
A flat vector bundle $V$ over $M$ admits a Hermitian--Einstein
metric if and only if $V$ is polystable. A polystable flat
vector bundle admits a unique Hermitian--Einstein connection.
\end{theorem}

Our aim here is to establish a similar result for flat principal
bundles over $M$.

Let $G$ be a connected Lie group such that it is either
a complex reductive linear algebraic group
or it is the fixed point locus of an anti-holomorphic
involution $\sigma_{G_{\mathbb C}}\, :\, G_{\mathbb C}\,
\longrightarrow\,
G_{\mathbb C}$, where $G_{\mathbb C}$ is a complex reductive
linear algebraic group; if
$G$ is of the second type, then we assume that
$\sigma_{G_{\mathbb C}}$ is of split type. Extending the
definition of a flat polystable vector bundle on $M$,
we define polystable flat principal $G$--bundles over $M$. A flat
principal $\text{GL}_r$--bundle is polystable if and only if the
corresponding vector bundle of rank $r$ is polystable.

Fix a maximal compact subgroup $K\, \subset\, G$. A Hermitian
structure on a flat principal $G$--bundle $E_G$ on $M$ is a
$C^\infty$ reduction of structure group $E_K\, \subset\, E_G$
to the subgroup $K$. If $G$ is the real points of
$G_{\mathbb C}$, given a flat
principal $G$--bundle $E_G$ over $M$, we get a holomorphic
principal $G_{\mathbb C}$--bundle $E_{G_{\mathbb C}}$
over the total space of $TM$.
For any Hermitian structure on $E_G$, there is a
naturally associated connection on the principal
$G_{\mathbb C}$--bundle $E_{G_{\mathbb C}}$.

A Hermitian structure on $E_G$ produces a connection on $E_G$.
Contracting using $g$ the curvature form of the connection, we
obtain a section of the adjoint bundle $\text{ad}(E_G)$. A Hermitian
structure on $E_G$ is called Hermitian--Einstein if this section of
$\text{ad}(E_G)$ is given by some element in the center of
$\text{Lie}(G)$. The connection associated to a Hermitian structure
satisfying this condition is called a Hermitian--Einstein
connection.

We prove the following generalization of Theorem \ref{thm0}
(see Theorem \ref{thm-2}):

\begin{theorem}\label{thm1}
A flat principal $G$--bundle $E_G\, \longrightarrow\, M$
admits a Hermitian--Einstein structure if and only if $E_G$
is polystable. A polystable flat principal $G$--bundle admits a
unique Hermitian--Einstein connection.
\end{theorem}

We prove the following
analog of the Harder--Narasimhan filtration (see Theorem
\ref{thm-1}):

\begin{theorem}\label{thm2}
For any flat vector bundle $V\,\longrightarrow\, M$, there
is a unique filtration of flat subbundles
$$
0\,=\, F_0\, \subset\, F_1\, \subset\, \cdots
\, \subset\, F_{\ell-1} \, \subset\, F_\ell\, =\, V
$$
such that for each $i\, \in\, [1\, ,\ell]$, the
flat vector bundle $F_i/F_{i-1}$ is semistable, and
$$
\mu_g(F_1) \, >\, \mu_g(F_2/F_1) \, >\, \cdots
\, >\, \mu_g(F_\ell/F_{\ell-1})\, .
$$
\end{theorem}

Our proof of Theorem \ref{thm1} crucially uses Theorem \ref{thm2}.

One goal of the present work and of \cite{Lo} is to develop analytic
tools to study affine manifolds. In the complex case, Li--Yau--Zheng
and also Teleman have used Hermitian--Einstein metrics on vector
bundles over non--K\"ahler surfaces equipped with Gauduchon metrics
to partially classify surfaces of Kodaira class VII
\cite{LYZ,Te94,Te05}. There are similar open problems in classifying
affine manifolds in low dimension. As suggested by Bill Goldman,
one case that may be tractable is that of flat affine symplectic
four--manifolds (flat affine four--manifolds admitting a flat
nondegenerate closed two--form). Little is known about these manifolds.

In particular, we hope to use the results of this paper to study
representations of $\pi_1(M)$ into a reductive Lie group $G$. It is
well known that a flat principal $G$--bundle over a manifold $M$ is
equivalent to a conjugacy class of homomorphisms from $\pi_1(M)$ to
$G$. Our Theorem \ref{thm1} thus can be rephrased as follows:
If $M$ is a special affine manifold equipped with a Gauduchon
metric $g$, every representation $\pi_1(M)\longrightarrow G$ admits
either a nontrivial destabilizing subrepresentation or a unique
Hermitian--Einstein connection.

For a general affine manifold $M$
and representation $\pi_1(M)\longrightarrow G$, we expect there to be
few subrepresentations at all, and so the existence of the canonical
Hermitian--Einstein connection is to be expected in many cases.

In Section \ref{sec.b}, we prove the
following Bogomolov type inequality as an
application of Theorem \ref{thm1} (see Lemma \ref{lem6}):

\begin{lemma}
Assume that the Gauduchon metric $g$ satisfies the condition
that $\partial\overline{\partial}(\omega^{d-2}_g)\, =\,0$,
where $\omega$ is the corresponding $(1\, ,1)$-form (it is
called an astheno--K\"ahler metric).
Let $V\, \longrightarrow\, M$ be a semistable flat vector bundle
of rank $r$. Then
$$
\int_M \frac{c_2({\mathcal E}nd(V))\omega^{d-2}_g}{\nu}\, =\,
\int_M \frac{(2r\cdot c_2(V) -
(r-1)c_1(V)^2)\omega^{d-2}_g}{\nu}\, \geq\, 0\, .
$$
\end{lemma}

A few notes about the proof are in order. As in \cite{Lo}, we are
able to follow the existing proofs in the complex case closely, with
a few important simplifications. In \cite{Lo}, the main
simplification is that we need only consider flat subbundles (as
opposed to singular subsheaves) as destabilizing objects in
non--stable vector bundles. In the current work, we find another such
simplification in the proof of Lemma \ref{lem2}. Flatness implies
that the space of all flat subbundles of rank $k$ of a given flat
vector bundle $V$ is a closed subset of the Grassmannian of the
fiber of $V$ at a given point in $M$, and thus a simple compactness
argument guarantees the existence of a flat subbundle with maximal
slope. The corresponding argument in the complex Gauduchon case is
more complicated \cite{Br}.

\section{Harder--Narasimhan filtration of a vector
bundle}\label{sec2}

We recall from \cite{Lo} some basic definitions. Consider a flat
affine manifold $M$ of dimension $n$ as the zero section of its
tangent bundle $TM$. The unifying idea behind all these definitions
is to consider flat objects on $M$ to be restrictions of holomorphic
objects on $TM$. We may define operators $\partial$ and
$\bar\partial$ on the affine Dolbeault complex, where $(p,q)$ forms
are represented as sections of $\Lambda^p(T^*M)\otimes
\Lambda^q(T^*M)$. A Riemannian metric $g$ on $M$ can be extended to
a natural Hermitian metric on the total space of $TM$. Let
$\omega_g$ be the associated $(1,1)$ form. The metric $g$ is called
\textit{affine Gauduchon} if
$\partial\bar\partial(\omega_g^{n-1})\,=\,0$. Given a flat vector bundle
$V\,\longrightarrow\,M$ equipped with a Hermitian bundle metric $h$,
we may define its first Chern form $c_1(h)$. Assume there is a
covariant constant volume form $\nu$ on $M$. Define the degree
of the vector bundle as
$$\deg_g(V)\,:=\,\int_M \frac{c_1(h)\wedge \omega_g^{n-1}}\nu\, .$$
Then the \textit{slope} of $V$ is defined to be
$$\mu_g(V) \,:=\, \deg_g(V)/{\rm rank}(V)\, .$$
The vector bundle $V$ is said to be
\textit{stable} if every flat subbundle $W$ of $V$
with $0\, <\, \text{rank}(W)\, <\, \text{rank}(V)$ satisfies the
inequality $\mu_g(W)\,<\,\mu_g(V)$. The vector bundle $V$ is called
\textit{semistable} if $\mu_g(W)\,\le\, \mu_g(V)$ for all such
$W$, and $V$ is said to be \textit{polystable} if it is a
direct sum of stable flat bundles of the same slope.

Let
$$
V\, \longrightarrow\, M
$$
be a flat vector bundle; it is allowed to be
real or complex. Fix a filtration of flat subbundles
\begin{equation}\label{e1}
0\,=\, V_0\, \subset\, V_1\, \subset\, \cdots
\, \subset\, V_{n-1} \, \subset\, V_n\, =\, V
\end{equation}
such that all the successive quotients $V_i/V_{i-1}$,
$1\, \leq\, i\, \leq\, n$, are stable.

Define
\begin{equation}\label{e2}
\delta\,:=\, \text{Max}\{\mu_g(V_i/V_{i-1})\}_{i=1}^n
\, \in\, \mathbb R\, .
\end{equation}

\begin{lemma}\label{lem1}
Let $F\, \subset\, V$ be any flat subbundle of $V$ of
positive rank. Then $\mu_g(F)\, \leq\, \delta$.
\end{lemma}

\begin{proof}
If $F$ is not semistable, then there is a flat subbundle
$F'\, \subset\, F$ such that
$$
0\, <\, \text{rank}(F')\, < \,\text{rank}(F) ~\,~\text{~and~}~\,~
\mu_g(F') \, > \, \mu_g(F)\, .
$$
Furthermore, if $F$ is semistable, then there is a flat subbundle
$F''\, \subset\, F$ such that $\text{rank}(F'')$
is smallest among the ranks of all flat subbundles $W$ of $F$ with
$\mu_g(W) \, = \, \mu_g(F)$. Note that such a smallest
rank flat vector
bundle $F''$ is automatically stable. Therefore,
it is enough to check the inequality in the lemma under the
assumption that $F$ is stable.

Assume that $F$ is stable, and $\mu_g(F)\, > \, \delta$.

Let $F_1$ and $F_2$ be semistable flat vector bundles over $M$ such that
either both of them are real or both are complex. Let
$$
\varphi\, :\, F_1\, \longrightarrow\, F_2
$$
be a nonzero flat homomorphism of vector bundles. Then
$$
\mu_g(F_1) \, \leq\, \mu_g(\text{Image}(\varphi))
\, \leq\,\mu_g(F_2)
$$
because $F_1$ and $F_2$ are semistable.
Therefore, there is no nonzero flat homomorphism of vector
bundles from $F_1$ to $F_2$ if $\mu_g(F_1)\, >\,
\mu_g(F_2)$.

{}From the above observation we conclude that for each $i\, \in
\, [1\, , n]$, there is
no nonzero flat homomorphism of vector bundles from $F$
to $V_i/V_{i-1}$. This immediately implies that
there is no nonzero flat homomorphism of vector bundles from $F$
to $V$. This contradicts the fact that $F$ is a
flat subbundle of $V$.
Therefore, we conclude that $\mu_g(F)\, \leq \, \delta$.
\end{proof}

Define
\begin{equation}\label{e3}
\delta_0(V)\,:=\, \text{Sup}\,\{\mu_g(F)\,\mid\, F
~\text{~is~a~flat~subbundle~of\,}~V\}
\end{equation}
which is a finite number due to Lemma \ref{lem1}.

\begin{lemma}\label{lem2}
There is a flat subbundle $F\, \subset\,V$ such that
$\mu_g(F)\,=\, \delta_0(V)$, where $\delta_0(V)$ is
defined in \eqref{e3}.
\end{lemma}

\begin{proof}
Fix an integer $k \, \in\, [1\, ,\text{rank}(V)]$ such that
$$
\delta_0(V)\,=\, \text{Sup}\,\{\mu_g(F)\,\mid\, F\, \subset\, V
~\text{~is~a~flat~subbundle~of~rank\,}~k\}\, ;
$$
such a $k$ clearly exists because $\text{rank}(V)$
is a finite integer. Fix a point
$$
x_0\, \in\, M\, .
$$
Let $\text{Gr}(V_{x_0},k)$ be the Grassmannian parametrizing
all linear subspaces of dimension $k$ of the fiber $V_{x_0}$.

For any flat subbundle $F\, \subset\, V$ of rank $k$, consider the
subspace $F_{x_0}\, \in\, \text{Gr}(V_{x_0},k)$. We note that the
flat subbundle $F$ is uniquely determined by the point $F_{x_0}\,
\in\, \text{Gr}(V_{x_0},k)$, because $M$ is connected. Let
\begin{equation}\label{e4}
S\, \subset\, \text{Gr}(V_{x_0},k)
\end{equation}
be the locus of all subspaces that are fibers of flat subbundles
of $V$ of rank $k$. We will now describe $S$ explicitly.

Let $$\rho\, :\, \pi_1(M, x_0)\, \longrightarrow\, \text{GL}(V_{x_0})$$
be the monodromy representation for the flat connection on $V$. The
group $\text{GL}(V_{x_0})$ acts on $\text{Gr}(V_{x_0},k)$ in a natural
way. The subset $S$ in \eqref{e4} is the fixed point locus
$$
S\, =\, \text{Gr}(V_{x_0},k)^{\rho(\pi_1(M, x_0))}
$$
for the action of $\rho(\pi_1(M, x_0))$ on $\text{Gr}(V_{x_0},k)$.
Note that
$$
S\,=\, \text{Gr}(V_{x_0},k)^{\overline{\rho(\pi_1(M, x_0))}}\, .
$$
This implies that $S$ is a closed subset of $\text{Gr}(V_{x_0},k)$.
In particular, $S$ is compact.

For each point $z\, \in\, S$, let $F^z\, \subset\, V$ be the unique
flat subbundle of $V$ such that $(F^z)_{x_0} \,=\, z$. We have a
continuous function
$$
f^k_V \, :\, S\, \longrightarrow\, \mathbb R
$$
defined by $z\, \longmapsto\, \mu_g(F^z)$. Since $S$ is compact,
there is a point $z_0\, \in\, S$ at which the function $f^k_V$
takes the maximum value $\delta_0(V)$. The corresponding flat
subbundle $F^{z_0}\, =\, F$ satisfies the condition in the lemma.
\end{proof}

\begin{proposition}\label{prop1}
There is a unique maximal flat subbundle $F\, \subset\, V$
with $\mu_g(F)\, =\, \delta_0(V)$, where $\delta_0(V)$ is
defined in \eqref{e3}.
\end{proposition}

\begin{proof}
Let $F_1$ and $F_2$ be two flat subbundles of $V$ such that
\begin{equation}\label{e4a}
\mu_g(F_1)\, =\, \mu_g(F_2)\,=\, \delta_0(V)\, .
\end{equation}
Note that $F_1$ and $F_2$ are automatically semistable. Let
$$
F_1+F_2\, \subset\, V
$$
be the flat subbundle generated by $F_1$ and $F_2$. We
have a short exact sequence of flat vector bundles
\begin{equation}\label{e5}
0\,\longrightarrow\, F_1\cap F_2\,\longrightarrow\,F_1\oplus
F_2\,\longrightarrow\,F_1+F_2\,\longrightarrow\, 0\, .
\end{equation}
Since $F_1$ and $F_2$ are both semistable, from \eqref{e4a} it
follows immediately that $F_1\oplus F_2$ is also semistable with
\begin{equation}\label{e6}
\mu_g(F_1\oplus F_2)\,=\, \delta_0(V)\, .
\end{equation}

We will show that $F_1+F_2$ is semistable with
$\mu_g(F_1 + F_2)\,=\, \delta_0(V)$.

To prove this, first note
that if $F_1\cap F_2\,=\, 0$, then $F_1\oplus F_2\,=\, F_1+ F_2$,
hence it is equivalent to the above observation. So assume that
$\text{rank}(F_1\cap F_2)\, >\, 0$.

{}From \eqref{e5},
$$
{\rm deg}_g(F_1\oplus F_2)\,=\, {\rm deg}_g(F_1\cap F_2)
+{\rm deg}_g(F_1 +F_2)\, .
$$
Hence
\begin{equation}\label{e6a}
\mu_g(F_1\oplus F_2)\,=\, \frac{\mu_g(F_1\cap F_2)\cdot
\text{rank}(F_1\cap F_2) + \mu_g(F_1 +F_2)\cdot
\text{rank}(F_1+ F_2)}{\text{rank}(F_1\cap F_2)+
\text{rank}(F_1+ F_2)}\, .
\end{equation}
Since $F_1\cap F_2$ and $F_1+ F_2$ are
flat subbundles of $V$, we have
\begin{equation}\label{e7}
\mu_g(F_1\cap F_2)\, , \mu_g(F_1+ F_2)
\, \leq\, \delta_0(V)\, .
\end{equation}
Using \eqref{e6} and \eqref{e6a} and \eqref{e7} we conclude that
$$
\mu_g(F_1+ F_2)\,=\, \mu_g(F_1\cap F_2)\, =\, \delta_0(V)\, .
$$

Therefore, we have proved that $F_1+F_2$ is semistable with
$\mu_g(F_1 + F_2)\,=\, \delta_0(V)$.

Consider the flat subbundle $F\, \subset\, V$ generated by all
flat subbundles $W$ with $\mu_g(W)\,=\, \delta_0(V)$.
Since $F_1+F_2$ is semistable with $\mu_g(F_1 + F_2)\,
=\, \delta_0(V)$ whenever $F_1$ and $F_2$ are
semistable with slope $\delta_0(V)$, if follows immediately
that the flat subbundle $F$ satisfies
the condition in the proposition.
\end{proof}

The unique maximal flat semistable subbundle $F\, \subset\, V$
with $\mu_g(F)\, =\, \delta_0(V)$ in Proposition \ref{prop1}
will be called the \textit{maximal semistable subbundle} of $V$.

Proposition \ref{prop1} has the following corollary:

\begin{corollary}\label{cor1}
There is a unique filtration of flat subbundles
$$
0\,=\, F_0\, \subset\, F_1\, \subset\, \cdots
\, \subset\, F_{\ell-1} \, \subset\, F_\ell\, =\, V
$$
such that for each $i\, \in\, [1\, ,\ell]$, the
flat subbundle $F_i/F_{i-1}\, \subset\, V/F_{i-1}$
is the unique maximal semistable subbundle.
\end{corollary}

The filtration in Corollary \ref{cor1} can be reformulated as
follows:

\begin{theorem}\label{thm-1}
Let $V$ be a flat vector bundle over $M$. Then there is a
unique filtration of $V$ by flat subbundles
$$
0\,=\, F_0\, \subsetneq\, F_1\, \subsetneq\, \cdots\,
\subsetneq\, F_j\, \subsetneq\, F_{j+1}\, \subsetneq\, \cdots
\, \subsetneq\, F_{\ell-1} \, \subsetneq\, F_\ell\, =\, V
$$
such that for each $i\, \in\, [1\, ,\ell]$, the
flat vector bundle $F_i/F_{i-1}$ is semistable, and
$$
\mu_g(F_1) \, >\, \mu_g(F_2/F_1) \, >\, \cdots
\, >\, \mu_g(F_{j+1}/F_j)\, >\, \cdots
\, >\, \mu_g(F_\ell/F_{\ell-1})\, .
$$
\end{theorem}

\begin{proof}
The filtration in Corollary \ref{cor1} clearly has the
property that for each $i\, \in\, [1\, ,\ell]$, the
flat subbundle $F_i/F_{i-1}$ is semistable, and
$$
\mu_g(F_1) \, >\, \mu_g(F_2/F_1) \, >\, \cdots
\, >\, \mu_g(F_\ell/F_{\ell-1})\, .
$$

Now, let
\begin{equation}\label{g1}
0\,=\, E_0\, \subset\, E_1\, \subset\, \cdots
\, \subset\, E_{n-1} \, \subset\, E_n\, =\, V
\end{equation}
be another filtration of flat subbundles of $V$
such that for each $i\, \in\, [1\, ,n]$, the
flat subbundle $E_i/E_{i-1}$ is semistable, and
$$
\mu_g(E_1) \, >\, \mu_g(E_2/E_1) \, >\, \cdots
\, >\, \mu_g(E_n/E_{n-1})\, .
$$
To show that the filtration in Corollary \ref{cor1}
coincides with the filtration in \eqref{g1}, it suffices
to prove that $E_1\,=\, F_1$, because we may replace
$V$ by $V/F_i$ and use induction on $i$.

If $n\, =\, 1$, then $V$ is semistable. Hence $F_1\, =\, V$,
and the theorem is evident.

Hence assume that $n\, \geq \, 2$.

We have
$$
\mu_g(E_n/E_{n-1}) \, <\, \mu_g(E_1)\, \leq\, \mu_g(F_1)
$$
because $F_1$ is the
maximal semistable subbundle of $V$. Therefore,
case there is no nonzero flat homomorphism from
$F_1$ to $E_n/E_{n-1}$ (see the proof of Lemma
\ref{lem1}). Now, by induction, there is no nonzero flat
homomorphism from $F_1$ to $E_i/E_{i-1}$ for all
$i\, \geq\, 2$. Hence there is no nonzero flat
homomorphism from $F_1$ to $V/E_1$.
Consequently, $F_1$ is a subbundle of $E_1$.

We have $\mu_g(E_1)\, \geq\, \mu_g(F_1)$ because
$E_1$ is semistable and $F_1$ is a subbundle of $E_1$.
On the other hand,
we have $\mu_g(F_1)\, \geq\, \mu_g(E_1)$
because $F_1$ is the
maximal semistable subbundle of $V$. Therefore,
$\mu_g(E_1)\, =\, \mu_g(F_1)$. Again from the
fact that $F_1$ is the maximal semistable
subbundle of $V$ we conclude that the subbundle
$F_1\, \subset\, E_1$ must coincide with $E_1$.
\end{proof}

\section{Semistability of tensor product}\label{sec3}

A flat vector bundle $(V\, ,D)$ over $M$ will be
called \textit{polystable} if
$$
(V\, ,D)\,=\, (\bigoplus_{i=1}^m W_i\, , \bigoplus_{i=1}^m D_i)\, ,
$$
where $(W_i\, ,D_i)$, $1\, \leq\, i\, \leq\, m$, are flat
stable vector bundles, and
$$
\mu_g(W_1)\,=\, \cdots\, =\, \mu_g(W_m)\, .
$$

Let $V_1$ and $V_2$ be flat vector bundles over $M$ such that
either both are real or both are complex.

\begin{lemma}\label{lem3}
If $V_1$ and $V_2$ are stable, then the flat vector bundle
$V_1\otimes V_2$ is polystable.
\end{lemma}

\begin{proof}
Assume that $V_1$ and $V_2$ are stable. Then each one them
admits an affine Hermitian--Einstein metric
(see \cite[p. 102, Theorem 1]{Lo} for the complex case and
\cite[p. 129, Corollary 33]{Lo} for the real case). Let $h_1$
and $h_2$ be affine Hermitian--Einstein metrics on
$V_1$ and $V_2$ respectively. The Hermitian metric on
$V_1\otimes V_2$ induced by $h_1$ and $h_2$ is clearly an
affine Hermitian--Einstein one. Therefore, $V_1\otimes V_2$ is
polystable \cite[p. 110, Theorem 4]{Lo}.
\end{proof}

\begin{corollary}\label{co2}
If $V_1$ and $V_2$ are polystable, then the flat vector bundle
$V_1\otimes V_2$ is also polystable.
\end{corollary}

\begin{proof}
This follows from Lemma \ref{lem3} after writing the flat
polystable vector bundles $V_1$ and $V_2$ as direct sums
of flat stable vector bundles.
\end{proof}

\begin{proposition}\label{prop2}
If $V_1$ and $V_2$ are semistable, then the flat vector bundle
$V_1\otimes V_2$ is semistable.
\end{proposition}

\begin{proof}
First assume that $V_2$ is stable. Let
\begin{equation}\label{e8}
0\,=\, W_0\, \subset\, W_1\, \subset\, \cdots
\, \subset\, W_{n-1} \, \subset\, W_n\, =\, V_1
\end{equation}
be a filtration of flat subbundles such that
each successive quotient $W_i/W_{i-1}$, $1\, \leq\, i\, \leq\, n$,
is stable with $\mu_g(W_i/W_{i-1})\,=\,\mu_g(V_1)$.
Consider the filtration of flat subbundles
\begin{equation}\label{re1}
0\,=\, W_0\otimes V_2\, \subset\, W_1\otimes V_2\, \subset\,
\cdots\, \subset\, W_{n-1}\otimes V_2 \, \subset\, W_n\otimes V_2
\, =\, V_1\otimes V_2
\end{equation}
obtained by tensoring the filtration in \eqref{e8} by $V_2$.
Each successive quotient in this filtration is polystable
by Lemma \ref{lem3}; also,
$$
\mu_g((W_i/W_{i-1})\otimes V_2)\,=\, \mu_g(W_i/W_{i-1})+
\mu_g(V_2)\, =\, \mu_g(V_1)+ \mu_g(V_2)\, .
$$
In view of these properties of the successive quotients
for the filtration in \eqref{re1} we conclude that
$V_1\otimes V_2$ is semistable.

If $V_2$ is not stable, then fix a filtration
$$
0\,=\, W'_0\, \subset\, W'_1\, \subset\, \cdots
\, \subset\, W'_{m-1} \, \subset\, W'_m\, =\, V_2
$$
such that each successive quotient
$W'_i/W'_{i-1}$, $1\, \leq\, i\, \leq\, m$,
is stable, and $\mu_g(W'_i/W'_{i-1})\,=\,\mu_g(V_2)$. Consider
the filtration of $V_1\otimes V_2$
\begin{equation}\label{e9}
0\,=\, V_1\otimes W'_0\, \subset\, V_1\otimes W'_1\, \subset\, \cdots
\, \subset\, V_1\otimes W'_{m-1} \, \subset\, V_1\otimes W'_m\, =\,
V_1\otimes V_2
\end{equation}
obtained by tensoring the above filtration with $V_1$.
For each $1\, \leq\, i\, \leq\, m$, the quotient
$$
(V_1\otimes W'_i)/(V_1\otimes W'_{i-1})\,=\,
V_1\otimes (W'_i/W'_{i-1})
$$
in \eqref{e9} is semistable by the earlier observation, and
furthermore,
$$
\mu_g(V_1\otimes (W'_i/W'_{i-1}))\,=\,
\mu_g(V_1)+\mu_g(W'_i/W'_{i-1})\,=\,\mu_g(V_1)+\mu_g(V_2)\, .
$$
Hence $V_1\otimes V_2$ is semistable.
\end{proof}

\begin{corollary}\label{cor0}
Let $V$ be a flat vector bundle over $M$. Take any integer
$j\, \in\, [1\, ,{\rm rank}(V)]$. If $V$ is polystable, then
the exterior power $\bigwedge^j V$ equipped with the induced
flat connection is polystable. If $V$ is semistable, then
$\bigwedge^j V$ equipped with the induced
flat connection is semistable.
\end{corollary}

\begin{proof}
If $V$ is polystable, then from Corollary \ref{co2} it follows
that $V^{\otimes j}$ equipped with the induced flat connection is
polystable. Since the flat vector bundle
$\bigwedge^j V$ is a direct summand of the flat vector bundle
$V^{\otimes j}$, we conclude that $\bigwedge^j V$ is polystable
if $V^{\otimes j}$ is so.

If $V$ is semistable, then from Proposition \ref{prop2} it follows
that $V^{\otimes j}$ equipped with the induced flat connection is
semistable. Therefore, the direct summand $\bigwedge^j V
\, \subset\, V^{\otimes j}$ is semistable.
\end{proof}

\section{Principal bundles on flat affine manifolds}\label{sec4}

\subsection{Preliminaries} Let $H$ be a Lie group. A principal
$H$--bundle on $M$ is a
triple of the form $(E_H\, ,p\, ,\psi)$, where $E_H$ is
a $C^\infty$ manifold, $p\, :\, E_H\, \longrightarrow\, M$
is a $C^\infty$ surjective submersion, and
$$
\psi\, :\, E_H\times H\, \longrightarrow\, E_H
$$
is a smooth action of $H$ on $E_H$, such that
\begin{enumerate}
\item $p\circ \psi\, =\, p\circ p_1$, where $p_1$ is the natural
projection of $E_H\times H$ to $E_H$, and

\item for each point $x\, \in\, M$, there is an open neighborhood
$U\, \subset\, M$ of $x$, and a smooth diffeomorphism
$$
\phi\, :\, p^{-1}(U) \, \longrightarrow\, U\times H\, ,
$$
such that $\phi$ commutes with the actions of $H$ (the group $H$
acts on $U\times H$ through right translations of $H$), and
$q_1\circ\phi\, =\, p$, where $q_1$ is the natural projection of
$U\times H$ to $U$.
\end{enumerate}

Let $dp\, :\, TE_H\, \longrightarrow\, p^*TM$ be the differential
of the projection $p$. A \textit{flat connection} on $E_H$ is
a $C^\infty$ homomorphism
$$
D\, :\, p^*TM\, \longrightarrow\,TE_H
$$
such that
\begin{itemize}
\item $dp\circ D\, =\, \text{Id}_{p^*TM}$,

\item the distribution $D(p^*TM)\, \subset\, TE_H$ is integrable, and

\item $D(p^*TM)$ is invariant under the action of $H$ on $TE_H$
given by the action of $H$ on $E_H$.
\end{itemize}

Let $H'\, \subset\, H$ be a closed subgroup. A \textit{reduction}
of structure group of a principal $H$--bundle $E_H$ to $H'$
is a principal $H'$--bundle $E_{H'}\, \subset\, E_H$; the action
of $H'$ on $E_{H'}$ is the restriction of the action of $H$ on
$E_H$. A reduction of structure group of $E_H$ to $H'$
is given by a smooth section of the fiber bundle
$E_H/H'\,\longrightarrow \, M$. We note that
if a reduction $E_{H'}\,\subset\, E_H$ corresponds to a section
$\sigma$, then $E'_H$ is the inverse image of $\sigma(M)$ for the
quotient map $E_H\, \longrightarrow\, E_H/H'$.

Let $D$ be a flat connection on $E_H$. A reduction of structure
group $E_{H'}\, \subset\, E_H$ to $H'$
is said to be \textit{compatible} with $D$ if for each point
$z\, \in \, E_{H'}$, the subspace $D(T_{p(z)}M)\, \subset\,
T_zE_H$ is contained in the subspace $T_z E_{H'}
\, \subset\, T_zE_H$. Note that this condition ensures
that $D$ produces a flat connection on $E_{H'}$.

Consider the adjoint action of $H$ on its Lie algebra
$\text{Lie}(H)$. Let
\begin{equation}\label{r-Ad}
\text{ad}(E_H) \, :=\, E_H\times^H \text{Lie}(H)\, \longrightarrow\, M
\end{equation}
be the vector bundle over $M$ associated to the principal $H$--bundle
$E_H$ for this action; it is known as the \textit{adjoint vector bundle}
for $E_H$. Since the adjoint action of $H$ on $\text{Lie}(H)$
preserves the Lie algebra structure, the fibers
of $\text{ad}(E_H)$
are Lie algebras isomorphic to $\text{Lie}(H)$. The connection $D$ on
$E_H$ induces
a connection on every fiber bundle associated to $E_H$. In
particular, $D$ induces a connection on the vector
bundle $\text{ad}(E_H)$; this induced connection on
$\text{ad}(E_H)$ will be denoted by $D^{\rm ad}$. The
connection $D^{\rm ad}$ is compatible with the Lie algebra structure
of the fibers of $\text{ad}(E_H)$, meaning
$$
D^{\rm ad}([s\, ,t]) \,=\, [D^{\rm ad}(s)\, ,t] +
[s\, ,D^{\rm ad}(t)]
$$
for all locally defined smooth sections $s$ and $t$ of
$\text{ad}(E_H)$.

\subsection{Stable and semistable principal bundles}\label{sec.s}

Let $G_{\mathbb C}$ be a complex reductive linear algebraic group.
A real form on $G_{\mathbb C}$ is an anti-holomorphic involution
$$
\sigma_{G_{\mathbb C}}\, :\, G_{\mathbb C}\,
\longrightarrow\, G_{\mathbb C}\, .
$$
The real form $\sigma_{G_{\mathbb C}}$ is said to be of
\textit{split type} if there is a maximal torus
$T\, \subset\, G_{\mathbb C}$ such that $\sigma_{G_{\mathbb C}}(T)
\, =\, T$ and the fixed point locus of the involution
$\sigma_{G_{\mathbb C}}\vert_T$ of $T$ is a
product of copies
of ${\mathbb R}^*$ (the group of nonzero real numbers).

Let $G$ be a connected Lie group such that either
it is a complex reductive linear algebraic group or it is the
fixed point locus of a split real form $\sigma_{G_{\mathbb C}}
\,\in\, \text{Aut}(G_{\mathbb C})$, where $G_{\mathbb C}$ and
$\sigma_{G_{\mathbb C}}$ are as above.

If $G$ is a complex reductive group,
a connected
closed algebraic subgroup $P\, \subset\, G$ is called a
\textit{parabolic} subgroup if the quotient variety $G/P$ is
complete. So, in particular, $G$ itself is a parabolic subgroup.
Let $P$ be a parabolic subgroup of $G$. A character
$\chi$ of $P$ is called \textit{strictly anti--dominant}
if the following two conditions hold:
\begin{itemize}
\item the line bundle over $G/P$ associated to the principal
$P$--bundle $G\, \longrightarrow\, G/P$ for $\chi$ is ample, and

\item the character $\chi$ is trivial on the connected component
of the center of $P$ containing the identity element.
\end{itemize}
Let $R_u(P)\, \subset\, P$ be the unipotent radical. The group
$P/R_u(P)$ is called the \textit{Levi quotient} of $P$.
A \textit{Levi subgroup} of $P$ is a connected reductive subgroup
$L(P)\, \subset\, P$ such that the composition
$$
L(P)\, \longrightarrow\, P\, \longrightarrow\,P/R_u(P)
$$
is an isomorphism. A Levi subgroup always exists
(see \cite[page 158, \S~11.22]{Bo} and
\cite[page 184, \S~30.2]{Hu}).

If $G$ is the fixed point locus of a real form
$(G_{\mathbb C}\, ,\sigma_{G_{\mathbb C}})$, by a parabolic
subgroup of $G$ we will mean a subgroup $P\, \subset\, G$ such that
there is a parabolic subgroup $P'\, \subset\, G_{\mathbb C}$
satisfying the conditions that $\sigma_{G_{\mathbb C}}(P')
\,=\, P'$ and $P'\bigcap G\, =\, P$. By a Levi subgroup of
the parabolic subgroup $P$ we will mean a subgroup $L(P)\, \subset\,
P$ such that there is a Levi subgroup $L(P')\, \subset\,
P'$ satisfying the conditions that $\sigma_{G_{\mathbb C}}(L(P'))
\,=\, L(P')$ and $L(P')\bigcap G\, =\, L(P)$.

Let $(E_G\, ,D)$ be a flat principal $G$--bundle over $M$.

It is called \textit{semistable} (respectively,
\textit{stable}) if for every triple of the form
$(Q\, , E_Q\, , \lambda)$, where
$Q\, \subset\, G$ is a proper parabolic
subgroup, and $E_Q\, \subset\, E_G$ is a reduction of structure
group of $E_G$ to $Q$ compatible with $D$, and
$\lambda$ is a strictly anti--dominant character of $Q$,
the inequality
\begin{equation}\label{es}
\text{deg}_g(E_Q(\lambda)) \, \geq\, 0
\end{equation}
(respectively, $\text{deg}_g(E_Q(\lambda)) \, > \, 0$) holds,
where $E_Q(\lambda)$ is the flat line bundle over $M$ associated to
the flat principal $Q$--bundle $E_Q$ for the character $\lambda$ of $Q$.

In order to decide whether $(E_G\, ,D)$ is
semistable (respectively, stable), it suffices to verify the
above inequality (respectively, strict inequality) only for
those $Q$ which are proper maximal parabolic subgroups of $G$.
More precisely, $E_G$ is semistable (respectively, stable)
if and only if for every pair $(Q\, , \sigma)$, where
$Q\, \subset\, G$ is a proper maximal parabolic subgroup,
and $\sigma\, :\, M\, \longrightarrow\, E_G/Q$
is a reduction of structure
group of $E_G$ to $Q$ compatible with $D$,
the inequality
\begin{equation}\label{st.eq}
\text{deg}_g(\sigma^* T_{\text{rel}})\, \geq\, 0
\end{equation}
(respectively, $\text{deg}_g(\sigma^* T_{\text{rel}})\,
>\, 0$) holds, where $T_{\text{rel}}$ is the relative
tangent bundle over $E_G/Q$ for the projection
$E_G/Q\, \longrightarrow\, M$. (See
\cite[page 129, Definition 1.1]{Ra} and \cite[page 131,
Lemma 2.1]{Ra}.) It should be mentioned that the connection
$D$ on $E_G$ induces a flat connection on the associated fiber
bundle $E_G/Q\, \longrightarrow\, M$. Since the section $\sigma$
is flat with respect to this induced connection (it is flat
because the reduction $E_Q$ is compatible with $D$), the
pullback $\sigma^* T_{\text{rel}}$ gets a flat connection.

Let $(E_G\, ,D)$ be a flat principal $G$--bundle over $M$. A reduction
of structure group
\[
E_Q\, \subset\, E_G
\]
to a parabolic subgroup $Q\, \subset\, G$ compatible with
$D$ is called \textit{admissible} if for each character $\lambda$ of
$Q$ trivial on the center of $G$, the associated flat line
bundle $E_Q(\lambda)\, \longrightarrow\, M$ satisfies the following
condition:
\begin{equation}\label{admiss}
\text{deg}_g(E_Q(\lambda)) \, =\, 0\, .
\end{equation}

We will call $(E_G\, ,D)$ to be
\textit{polystable} if either $E_G$ is stable, or there
is a proper parabolic subgroup $Q$ and a reduction of
structure group $E_{L(Q)}\, \subset\, E_H$ to
a Levi subgroup $L(Q)$ of $Q$ compatible with $D$ such that
the flat principal $L(Q)$--bundle $E_{L(Q)}$ is stable, and
the reduction of structure group of $E_G$ to $Q$,
obtained by extending the structure group of $E_{L(Q)}$
using the inclusion of $L(Q)$ in $Q$, is admissible.

We note that a flat polystable principal $G$--bundle on $M$
is semistable.

For notational convenience, we will omit the symbol of
connection for a flat principal bundle. When we will say
``$E_G$ be a flat principal $G$--bundle'' it will mean that
$E_G$ is equipped with a flat connection.

\subsection{Harder--Narasimhan reduction of principal bundles}

Let $G$ be as before. Let $E_G$ be
a flat principal $G$--bundle over $M$.

A \textit{Harder--Narasimhan reduction} of $E_G$ is a
pair of the form $(P\, ,E_P)$, where $P\, \subset\, G$ is
a parabolic subgroup, and $E_P\, \subset\, E_G$ is a reduction
of structure group of $E_G$ to $P$ compatible with the connection
such that the following two conditions hold:
\begin{enumerate}
\item The principal $P/R_u(P)$--bundle $E_P/R_u(P)$
equipped with the induced flat connection is semistable, where
$R_u(P)\, \subset\, P$ is the unipotent radical.

\item For any nontrivial character $\chi$ of $P$ which can be
expressed as a nonnegative integral combination of simple roots,
the flat line bundle over $M$ associated to $E_P$ for $\chi$ is of
positive degree.
\end{enumerate}

\begin{proposition}\label{prop3}
A flat principal $G$--bundle $E_G$ admits a Harder--Narasimhan
reduction. If $(P\, ,E_P)$ and $(P\, ,E_P)$ are two Harder--Narasimhan
reductions of $E_G$, then there is an element $g\,\in\, G$ such that
$Q\,=\, g^{-1}Pg$ and $E_Q\,=\, E_Pg$.
\end{proposition}

\begin{proof}
Let $\text{ad}(E_G)\, \longrightarrow\, M$ be the adjoint vector
bundle of $E_G$ (defined in \eqref{r-Ad}). As mentioned earlier,
the flat connection on $E_G$ induces a flat connection
on $\text{ad}(E_G)$. Let
\begin{equation}\label{fl}
0\,=\, V_0\, \subset\, V_1\,\subset\, \cdots\,\subset\,
V_{\ell-1}\, \subset\, V_\ell\,=\, \text{ad}(E_G)
\end{equation}
be the Harder--Narasimhan filtration of the flat vector bundle
$\text{ad}(E_G)$ constructed in Theorem \ref{thm-1}. Using
Proposition \ref{prop2} it can be deduced that $\ell$ in
\eqref{fl} is an odd integer; its proof is identical to the
proof of (3) in \cite[p. 215]{AB}. The flat subbundle
\begin{equation}\label{el1}
V_{(\ell+1)/2}\, \subset\, \text{ad}(E_G)
\end{equation}
in \eqref{fl}
is the adjoint vector bundle of a reduction of structure group of
$E_G$ to a parabolic subgroup; its proof is identical to the
proof of \cite[p. 699, Lemma 4]{AAB}. After we fix a parabolic
subgroup $P$ of $G$ in the conjugacy class of parabolic subgroups
of $G$ defined by the subalgebra $(E_{(\ell+1)/2})_x\, \subset\,
\text{ad}(E_G)_x$, where $x\, \in\, M$, we get a reduction of
structure group of $E_G$ to $P$ compatible with the connection.
This reduction satisfies all the conditions in the proposition.
The details of the argument are in \cite{AAB}.
\end{proof}

The second condition in the Harder--Narasimhan reduction can be
reformulated in other equivalent ways; see \cite{AAB}.

\begin{remark}
{\rm Proposition \ref{prop3} can also be proved by imitating
the proof of Proposition 3.1 in \cite{BH}.}
\end{remark}

Proposition \ref{prop3} has the following corollary:

\begin{corollary}\label{cor2}
A flat principal $G$--bundle $E_G$ over $M$ is semistable if and
only if the flat vector bundle ${\rm ad}(E_G)$ is semistable.
\end{corollary}

\begin{proof}
Assume that $\text{ad}(E_G)$ is not semistable. Then $E_{(\ell+1)/2}$
in \eqref{fl} is a proper subbundle of $\text{ad}(E_G)$. Hence
$E_G$ has a nontrivial Harder--Narasimhan reduction $(P\, , E_P)$.
Let $\mathfrak g$ and $\mathfrak p$ be the Lie algebras of $G$
and $P$ respectively. The group $P$
has the adjoint action on ${\mathfrak g}/\mathfrak p$.
The vector
bundle over $M$ associated to the principal $P$--bundle $E_P$
for the $P$--module ${\mathfrak g}/\mathfrak p$ is identified
with the vector bundle $\text{ad}(E_G)/E_{(\ell+1)/2}$. Consequently,
the reduction $E_P\, \subset\, E_G$ and the strictly
anti--dominant character of $P$ defined by the $P$--module
$\bigwedge^{\rm top} ({\mathfrak g}/\mathfrak p)$ violate
the inequality in \eqref{es}. Hence the flat principal $G$--bundle
$E_G$ is not semistable.

To prove the converse,
assume that the flat vector bundle $\text{ad}(E_G)$ is semistable.
Then $E_{(\ell+1)/2}\, =\, \text{ad}(E_G)$ (see \eqref{el1}).
Hence the Harder--Narasimhan reduction of $E_G$ is
$(G\, ,E_G)$ itself. Since the Levi quotient of $G$ is $G$
itself, from the first condition in the definition of a
Harder--Narasimhan reduction we conclude that $E_G$ is semistable.
\end{proof}

\begin{corollary}\label{cor-n1}
Assume that $G$ is the fixed point locus of a split real form
on $G_{\mathbb C}$. Let $E_G$ be a flat principal $G$--bundle
over $M$. Let $E_{G_{\mathbb C}}$
be the flat principal $G_{\mathbb C}$--bundle
over $M$ obtained by extending the structure group
of $E_G$ using the inclusion of $G$ in $G_{\mathbb C}$.
The principal $G$--bundle $E_G$ is semistable if and only
if the principal $G_{\mathbb C}$--bundle $E_{G_{\mathbb C}}$ is so.
\end{corollary}

\begin{proof}
Let $V$ be a flat real vector bundle over $M$. Let
$V_{\mathbb C}\,:=\, V\bigotimes_{\mathbb R}\mathbb C$ be the
flat complex vector bundle. We will show that $V$ is semistable
if and only if $V_{\mathbb C}$ is so.

If a flat subbundle $W\, \subset\, V$ violates the semistability
condition for $V$, then the flat subbundle $W\bigotimes_{\mathbb
R}\mathbb C\, \subset\, V_{\mathbb C}$ violates the semistability
condition for $V_{\mathbb C}$. Therefore, $V$ is semistable if
$V_{\mathbb C}$ is so.

To prove the converse, assume that
$V_{\mathbb C}$ is not semistable. Let $F\, \subset\,
V_{\mathbb C}$ be the maximal semistable subbundle of
$V_{\mathbb C}$, which is a proper subbundle because $V_{\mathbb C}$
is not semistable. From the uniqueness of $F$ it follows immediately
that the $\mathbb R$--linear conjugation automorphism of
$V_{\mathbb C}\, =\, V\bigotimes_{\mathbb R}\mathbb C$
defined by $v\otimes \lambda\, \longmapsto\, v\otimes
\overline{\lambda}$, where $v\,\in\, V$ and $\lambda
\,\in\, \mathbb C$, preserves the subbundle $F$. Hence $F$ is the
complexification of a flat subbundle $F'$ of $V$. This subbundle
$F'$ violates the semistability condition for $V$.
Therefore, $V_{\mathbb C}$ is semistable if and only if $V$ is so.

We apply the above observation to $V\,=\, \text{ad}(E_G)$.
Note that
\begin{equation}\label{e-l1}
\text{ad}(E_{G_{\mathbb C}})\,=\,
\text{ad}(E_G)\otimes_{\mathbb R}\mathbb C\, .
\end{equation}
In view of Corollary \ref{cor2}, the proof is complete.
\end{proof}

\section{The socle reduction}

Let $V\, \longrightarrow\, M$ be a flat semistable vector bundle;
it is allowed to be real or complex.
Let $F_1$ and $F_2$ be two flat subbundles of $V$ such that
both $F_1$ and $F_2$ are polystable, and
\begin{equation}\label{e10}
\mu_g(F_1)\,=\, \mu_g(F_2) \,=\, \mu_g(V)\, .
\end{equation}
Let
$$
F_1+ F_2\, \subset\, V
$$
be the flat subbundle of $V$ generated by $F_1$ and $F_2$.

\begin{proposition}\label{prop4}
The flat vector bundle $F_1+ F_2$ is polystable, and
$\mu_g(F_1+ F_2)\,=\, \mu_g(V)$.
\end{proposition}

\begin{proof}
{}From \eqref{e10},
\begin{equation}\label{e11}
\mu_g(F_1\oplus F_2)\,=\, \mu_g(V)\, .
\end{equation}
Consider the short exact sequence of flat vector bundles
\begin{equation}\label{f1}
0\,\longrightarrow\, F_1\cap F_2\,\longrightarrow\,F_1\oplus
F_2\,\longrightarrow\,F_1+F_2\,\longrightarrow\, 0\, .
\end{equation}

If $F_1\cap F_2\,=\, 0$, then $F_1+F_2\,=\, F_1\oplus F_2$, hence
in this case
$F_1+F_2$ is polystable, and $\mu_g(F_1+F_2)\,=\, \mu_g(V)$ from
\eqref{e11}. Therefore, the proposition is evident if
$F_1\cap F_2\,=\, 0$.

So assume that $\text{rank}(F_1\cap F_2)\, >\, 0$.

Since $V$ is semistable, and both $F_1\cap F_2$ and
$F_1+F_2$ are flat subbundles of $V$, we have
\begin{equation}\label{inq}
\mu_g(F_1\cap F_2)\, , \mu_g(F_1+ F_2)\, \leq\,
\mu_g(V)\, .
\end{equation}
{}From \eqref{f1},
$$
\mu_g(F_1\oplus F_2)\,=\, \frac{\mu_g(F_1\cap F_2)\cdot
\text{rank}(F_1\cap F_2) + \mu_g(F_1 +F_2)\cdot
\text{rank}(F_1+ F_2)}{\text{rank}(F_1\cap F_2)+
\text{rank}(F_1+ F_2)}\, .
$$
Combining this with \eqref{e11} and \eqref{inq},
\begin{equation}\label{e12}
\mu_g(F_1+ F_2)\, =\, \mu_g(F_1\cap F_2)\, =\,\mu_g(V)\, .
\end{equation}

Let $r$ be the rank of the flat vector bundle $F_1\cap F_2$;
recall that it is positive. Consider the vector bundle
\begin{equation}\label{W}
{\mathcal W}\, :=\, {\mathcal H}om(\bigwedge\nolimits^r (F_1\cap
F_2)\, , \bigwedge\nolimits^r F_1)\,=\, \bigwedge\nolimits^r
(F_1\cap F_2)^* \otimes \bigwedge\nolimits^r F_1\, .
\end{equation}
Note that the inclusion homomorphism $F_1\cap F_2\,\hookrightarrow
\, F_1$ defines a nonzero flat section
\begin{equation}\label{e13}
\eta\, \in\, H^0(M,\, {\mathcal W})\, .
\end{equation}

We have
$$
\mu_g({\mathcal W})\,= \, \mu_g(\bigwedge\nolimits^r F_1)
- \mu_g(\bigwedge\nolimits^r (F_1\cap F_2)) \,=\,
r\cdot \mu_g(F_1) - r\cdot \mu_g(F_1\cap F_2)\, .
$$
Hence $\mu_g({\mathcal W})\,= \, 0$ by \eqref{e10} and \eqref{e12}.
Hence, ${\rm deg}_g({\mathcal W})\,=\, 0$; also, from
Corollary \ref{cor0} we know that ${\mathcal W}$ is polystable
(note that $\bigwedge\nolimits^r(F_1\cap F_2)^*$ is a line bundle).
We recall from \cite{Lo} that given any flat vector
bundle $V$ on $M$ of degree zero equipped with a
Hermitian--Einstein connection
$\nabla_V$, any flat section of $V$ is flat with respect to
$\nabla_V$ (see Theorem 3 of \cite[p. 110]{Lo}).
Also, \cite[p. 102, Theorem 1]{Lo} and
\cite[p. 129, Corollary 33]{Lo} say that
any polystable vector bundle on $M$ admits a
Hermitian--Einstein connection. Hence the vector bundle
${\mathcal W}$ in \eqref{W} admits a Hermitian--Einstein connection,
and the section $\eta$ in \eqref{e13} is flat with respect to
the Hermitian--Einstein connection on ${\mathcal W}$.

Since $\eta$ is flat with respect to
the Hermitian--Einstein connection on ${\mathcal W}$,
it follows that
the Hermitian--Einstein connection on $F_1$ preserves
the subbundle $F_1\cap F_2\, \subset\, F_1$. Consequently,
$F_1\cap F_2$ is polystable \cite[p. 110, Theorem 4]{Lo}. This
also implies that the orthogonal complement of $F_1\cap F_2$
with respect to a Hermitian--Einstein metric on $F_1$
$$
F'\, :=\, (F_1\cap F_2)^\perp \, \subset\, F_1
$$
is preserved by the Hermitian--Einstein connection. Hence
$F'$ is polystable if $F'\,\not=\, 0$; note that
$\mu_g(F')\,=\, \mu_g(F_1)$ if $F'\,\not=\, 0$.

Since $F_1 +F_2\, =\, F'\oplus F_2$, we now conclude that
$F_1 +F_2$ is polystable, and $\mu_g(F_1+ F_2)\,=\, \mu_g(V)$.
\end{proof}

\begin{corollary}\label{cor3}
Let $V\, \longrightarrow\, M$ be a flat semistable vector bundle.
Then there is a unique maximal polystable flat subbundle
$F\, \subset\, V$ such that $\mu_g(F)\,=\, \mu_g(V)$.
\end{corollary}

\begin{proof}
In view of Proposition \ref{prop4}, the flat subbundle $F\, \subset
\, V$ generated by all flat polystable subbundles $E\, \subset
\, V$ with $\mu_g(E)\,=\, \mu_g(V)$ satisfies the conditions
in the corollary.
\end{proof}

The flat polystable subbundle $F\, \subset\, V$ in Corollary
\ref{cor3} is called the \textit{socle} of $V$.

If $F$ is properly contained in $V$, then we note
that $V/F$ is semistable, and $\mu_g(V/F)\,=\,\mu_g(V)$.
Therefore, Corollary \ref{cor3} gives the following:

\begin{corollary}\label{cor4}
Let $V\, \longrightarrow\, M$ be a flat semistable vector bundle.
Then there is a filtration of flat subbundles
$$
0\,=\, F_0\, \subset\, F_1\, \subset\, \cdots
\, \subset\, F_{n-1} \, \subset\, F_n\, =\, V
$$
such that for each $i\, \in\, [1\, ,n]$, the flat
subbundle $F_i/F_{i-1} \, \subset\, V/F_{i-1}$
is the socle of the flat semistable vector bundle $V/F_{i-1}$.
\end{corollary}

\subsection{Socle reduction of a principal bundle}

Let $G$ be as before.
Let $E_G\, \longrightarrow\, M$ be a semistable principal $G$--bundle.

A \textit{socle reduction} of $E_G$ is a pair
$(Q_0\, ,E_{Q_0})$, where
\begin{itemize}
\item $Q_0\, \subset\, H$
is maximal among all the parabolic
subgroups $Q$ of $G$ such that $E_G$ admits an admissible
reduction of structure group
$$
E_Q\, \subset\, E_G
$$
for which the corresponding principal $Q/R_u(Q)$--bundle
$E_Q/R_u(Q)\, \longrightarrow\, M$ is polystable, where
$R_u(Q)$ is the unipotent radical of $Q$, and

\item $E_{Q_0}\, \subset\, E_G$ is an admissible reduction
of structure group of $E_G$ to $Q_0$ such that the associated
principal $Q_0/R_u(Q_0)$--bundle $E_{Q_0}/R_u(Q_0)$
is polystable.
\end{itemize}
(Admissible reductions were defined in \eqref{admiss}.)

\begin{proposition}\label{prop5}
Let $E_G\, \longrightarrow\, M$ be a semistable principal
$G$--bundle. Then $E_G$ admits a socle reduction. If
$(Q_1\, ,E_{Q_1})$ and $(Q_2\, ,E_{Q_2})$ are two socle
reductions of $E_G$, then there is an element $g\,\in\, G$ such
that $Q_2\,=\, g^{-1}Q_1g$ and $E_{Q_2}\,=\, E_{Q_1}g$.
\end{proposition}

\begin{proof}
{}From Corollary \ref{cor2} we know that the flat adjoint
bundle ${\rm ad}(E_G)$ is semistable. Let
\begin{equation}\label{er}
0\,=\, F_0\, \subset\, F_1\, \subset\, \cdots\,
\subset\, F_{n-1} \, \subset\, F_n\, =\, {\rm ad}(E_G)
\end{equation}
be the filtration constructed in Corollary \ref{cor4}. Using
Corollary \ref{co2} it can be shown that $n$ is an odd integer;
see \cite[p. 218]{AB} for the details.
The flat subbundle
\begin{equation}\label{el2}
F_{(n+1)/2}\, \subset\, \text{ad}(E_G)
\end{equation}
in \eqref{er}
is the adjoint vector bundle of a reduction of structure group of
$E_G$ to a parabolic subgroup. After we fix a parabolic
subgroup $Q_0$ of $G$ in the conjugacy class of parabolic subgroups
of $G$ defined by the subalgebra $(E_{(n+1)/2})_x\, \subset\,
\text{ad}(E_G)_x$, where $x\, \in\, M$, we get a reduction of
structure group $E_{Q_0}\, \subset\, E_G$ to $Q_0$
compatible with the connection on $E_G$.
It can be shown that this pair $(Q_0\, ,E_{Q_0})$ is a socle
reduction of $E_G$ \cite{AB}. The uniqueness statement is also
proved in \cite{AB}.
\end{proof}

{}From Proposition \ref{prop5} and its proof we have the following
corollary:

\begin{corollary}\label{cor5}
Let $E_G$ be a flat principal $G$--bundle over $M$. Then
$E_G$ is polystable if and only if the flat vector bundle ${\rm
ad}(E_G)$ is polystable.
\end{corollary}

\begin{proof}
First we assume that $E_G$ is polystable. Then
the flat vector bundle ${\rm ad}(E_G)$ is semistable by
Corollary \ref{cor2}. If ${\rm ad}(E_G)$ is not polystable,
and $(Q_0\, ,E_{Q_0})$ is a socle
reduction of $E_G$, then $Q_0$ is a proper parabolic subgroup
of $G$. Therefore, $E_G$ is not polystable, which contradicts
the assumption. Hence ${\rm ad}(E_G)$ is polystable.

To prove the converse, assume that ${\rm ad}(E_G)$ is polystable.
Then the principal $G$--bundle $E_G$ is semistable (see Corollary
\ref{cor2}). Consider the socle of $E_G$. Since ${\rm ad}(E_G)$ is
polystable, we have $E_{(n+1)/2}\, =\, \text{ad}(E_G)$ (see
\eqref{el2}). Hence $(G\, ,E_G)$ is the socle of $E_G$.
Therefore, from the definition of a socle we conclude that
$E_G$ is polystable.
\end{proof}

\begin{corollary}\label{cor-n2}
Assume that $G$ is the fixed point locus of a split real form
on $G_{\mathbb C}$. Let $E_G$ be a flat principal $G$--bundle
over $M$. Let $E_{G_{\mathbb C}}$
be the flat principal $G_{\mathbb C}$--bundle
over $M$ obtained by extending the structure group
of $E_G$ using the inclusion of $G$ in $G_{\mathbb C}$.
The principal $G$--bundle $E_G$ is polystable if and only
if the principal $G_{\mathbb C}$--bundle $E_{G_{\mathbb C}}$ is so.
\end{corollary}

\begin{proof}
We note that $E_G$ is polystable if and only if
$\text{ad}(E_G)$ is polystable by Corollary \ref{cor5}. Since
$$
\text{ad}(E_{G_{\mathbb C}})\, =\, \text{ad}(E_G)\otimes_{\mathbb R}
{\mathbb C}\, = \, \text{ad}(E_G)\oplus \sqrt{-1}\cdot
\text{ad}(E_G)\, ,
$$
it follows that $\text{ad}(E_G)$ is polystable if and only if
$\text{ad}(E_{G_{\mathbb C}})$ is polystable. From
Corollary \ref{cor5}, the adjoint
vector bundle $\text{ad}(E_{G_{\mathbb C}})$
is polystable if and only if $E_{G_{\mathbb C}}$ is polystable.
\end{proof}

\section{Hermitian--Einstein connection on stable bundles
and Bogomolov inequality}

\subsection{Hermitian--Einstein connection and stable principal
bundles}

Fix a maximal compact subgroup
$$
K\, \subset\, G
$$
of the reductive group $G$. Let $E_G$ be a flat principal $G$--bundle
over $M$. A \textit{Hermitian structure} on $E_G$ is a $C^\infty$
reduction of structure group
$$
E_K\, \subset\, E_G\, .
$$

Recall that $G$ is either the fixed point locus of a split
real form on a complex reductive group $G_{\mathbb C}$ or $G$ is
complex reductive. In the second case, by $G_{\mathbb C}$ we will
denote $G$ itself; this is for notational convenience.

Given a flat
principal $G$--bundle $E_G$ over $M$, we get a holomorphic
principal $G_{\mathbb C}$--bundle $E_{G_{\mathbb C}}$
over $M_{\mathbb C}$ (the total
space of $TM$); see \cite[p. 102]{Lo}.

Given a Hermitian structure on $E_G$, there is a naturally
associated connection on the principal $G_{\mathbb C}$--bundle
$E_{G_{\mathbb C}}$ over $M_{\mathbb C}$
\cite[p. 106, Lemma 1]{Lo}; although \cite[Lemma 1]{Lo} is
only for vector bundles, the proof for principal bundles
is identical.

Any element $z$ of the center of the Lie algebra
${\mathfrak g}\, =\, \text{Lie}(G)$ defines a flat
section of $\text{ad}(E_G)$, because $z$ is fixed
by the adjoint action of $G$ on ${\mathfrak g}$; this section of
$\text{ad}(E_G)$ given by $z$ will be denoted by $\underline{z}$.

Let $E_K\, \subset\, E_G$ be a Hermitian structure on $E_G$.
Let $\nabla$ be the corresponding connection on $E_{G_{\mathbb C}}$.
The curvature of $\nabla$ will be denoted by $K(\nabla)$. So
$K(\nabla)$ is a smooth $(1\, ,1)$--form on $M_{\mathbb C}$
with values in the adjoint vector bundle $\text{ad}
(E_{G_{\mathbb C}})$. Contracting it using the metric $g$, we get
a smooth section $\Lambda_g K(\nabla)$ of $\text{ad}(E_G)$.
The Hermitian structure $E_K$ is called \textit{Hermitian--Einstein}
if there is an element $z$ in the center of the Lie algebra
$\mathfrak g$ such that
$$
\Lambda_g K(\nabla)\, =\, \underline{z}\, .
$$

If $E_K$ is a Hermitian--Einstein structure, then the
corresponding connection
$\nabla$ is called a \textit{Hermitian--Einstein connection}.

\begin{theorem}\label{thm-2}
A flat principal $G$--bundle $E_G\, \longrightarrow\, M$
admits a Hermitian--Einstein structure if and only if $E_G$
is polystable. A polystable flat principal $G$--bundle admits a
unique Hermitian--Einstein connection.
\end{theorem}

\begin{proof}
First we assume that $E_G$ admits a Hermitian--Einstein structure.
A Hermitian--Einstein structure on $E_G$ induces a Hermitian--Einstein
metric on the adjoint vector bundle $\text{ad}(E_G)$. Hence
$\text{ad}(E_G)$ is polystable
\cite[p, 110, Theorem 4]{Lo}. Hence $E_G$ is polystable by
Corollary \ref{cor5}.

To prove the converse, assume $E_G$ is polystable. We will
first reduce to the case that $G$ is complex reductive.

If $G$ is the fixed point locus of a split
real form on $G_{\mathbb C}$, then
Corollary \ref{cor-n2} says that the corresponding
principal $G_{\mathbb C}$--bundle $E_{G_{\mathbb C}}$
is also polystable. In the following paragraphs, we produce a unique
Hermitian--Einstein connection on $E_{G_{\mathbb C}}$. Uniqueness implies
that it is invariant under a natural complex conjugation, as in the
proof of Corollary
\ref{cor-n1} above, and so the Hermitian--Einstein connection on
$E_{G_{\mathbb C}}$ reduces to a connection on $E_G$.

We assume that $G$ is complex reductive.
Let $Z_G\, \subset\, G$ be the center of $G$. The adjoint
action of $G/Z_G$ on $\mathfrak g$ is faithful. Since $G$ is
reductive, the quotient $G/[G\, ,G]$ is a product of copies
of ${\mathbb C}^*$.

Let $E_G\, \longrightarrow\, M$ be a flat polystable principal
$G$--bundle. Then the vector bundle $\text{ad}(E_G)$ is polystable
by Corollary \ref{cor5}. Let $\nabla({\rm ad})$ be the
Hermitian--Einstein connection on $\text{ad}(E_G)$.

We will show that the connection $\nabla({\rm ad})$ preserves
the Lie algebra structure of the fibers of $\text{ad}(E_G)$.

Let
$$
\theta_0\, :\, \text{ad}(E_G)\otimes \text{ad}(E_G)\,
\longrightarrow\, \text{ad}(E_G)
$$
be the homomorphism defined by the Lie algebra structure
of the fibers of $\text{ad}(E_G)$. Define the flat vector bundle
\begin{equation}\label{cW}
{\mathcal W}\,:=\,
{\mathcal H}om(\text{ad}(E_G)\otimes \text{ad}(E_G)\, ,
\text{ad}(E_G))\,=\, (\text{ad}(E_G)\otimes \text{ad}(E_G))^*
\otimes \text{ad}(E_G)\, .
\end{equation}
Let
\begin{equation}\label{theta}
\theta\, \in\, C^\infty(M,\, {\mathcal W})
\end{equation}
be the smooth section defined by the above homomorphism $\theta_0$.
We note that $\theta$ is flat with respect to the flat
connection on $\mathcal W$ induced by the flat connection on
$E_G$.

Fix a nondegenerate $G$--invariant symmetric bilinear form $B$ on
$\mathfrak g$; such a form exists because $G$ is either complex
reductive or the fixed point locus of a real form of a
complex reductive group.Since $B$ is $G$--invariant, it produces a
symmetric bilinear form $\widetilde B$ on $\text{ad}(E_G)$
which is fiberwise nondegenerate and is preserved by the flat
connection on $\text{ad}(E_G)$. Therefore, $\widetilde B$
produces an isomorphism of the flat vector bundle $\text{ad}(E_G)$
with its dual $\text{ad}(E_G)^*$. Hence
${\rm deg}_g(\text{ad}(E_G))\, =\,-{\rm deg}_g(\text{ad}(E_G)^*)
\, =\, -{\rm deg}_g(\text{ad}(E_G))$, implying that
$${\rm deg}_g(\text{ad}(E_G))\,=\, 0\, .$$
Hence
\begin{equation}\label{d1}
{\rm deg}_g({\mathcal W})\,=\, 0\, ,
\end{equation}
where ${\mathcal W}$ is defined in \eqref{cW}.

The Hermitian--Einstein
connection $\nabla({\rm ad})$ on $\text{ad}(E_G)$ induces
a connection on the vector bundle ${\mathcal W}$;
this induced connection will be denoted by $\widetilde{\nabla}$.
Since the connection $\nabla({\rm ad})$ is Hermitian--Einstein,
the connection $\widetilde{\nabla}$ is also
Hermitian--Einstein. Therefore, from \eqref{d1} it follows that
any flat section of ${\mathcal W}$ is also flat with respect
to the Hermitian--Einstein connection $\widetilde{\nabla}$
\cite[p. 110, Theorem 3]{Lo}. In particular, the section
$\theta$ in \eqref{theta} is flat with respect to
$\widetilde{\nabla}$. This immediately implies that
the connection $\nabla({\rm ad})$ preserves
the Lie algebra structure of the fibers of $\text{ad}(E_G)$.

Hence the connection $\nabla({\rm ad})$ produces a connection
on the flat principal $G/Z_G$--bundle $E_G/Z_G$. This connection
on $E_G/Z_G$ will be denoted by $\nabla'$. The connection
$\nabla'$ is Hermitian--Einstein because $\nabla({\rm ad})$ is so.

Since $G/[G\, ,G]\,=\, ({\mathbb C}^*)^d$, where $d$ is the
dimension of $Z_G$, the flat principal $G/[G\, ,G]$--bundle
$E_G/[G\, ,G]$ admits a unique Hermitian--Einstein connection;
this connection will be denoted by $\nabla''$.

The quotient map $G\, \longrightarrow\, (G/Z_G)\times (G/[G\, ,G])$
has the property that the corresponding homomorphism of Lie algebras
is an isomorphism. Therefore, there is a natural bijection between
the connections on a principal $G$--bundle $F_G$ and the
connections on the principal $G/Z_G)\times (G/[G\, ,G])$--bundle
obtained by extending the structure group of $F_G$ using the
above homomorphism $G\, \longrightarrow\, (G/Z_G)\times (G/[G\, ,G])$.
The connections $\nabla'$ and $\nabla''$ on $E_G/Z_G$ and $E_G/[G\, ,G]$
together define a connection on the principal
$G/Z_G)\times (G/[G\, ,G])$--bundle $(E_G/Z_G)\times_M (E_G/[G\, ,G])$
over $M$. By the above remark on bijection of connections, this
connection on $(E_G/Z_G)\times_M (E_G/[G\, ,G])$ produces a connection
on $E_G$. The connection on $E_G$ obtained this way is
Hermitian--Einstein because both $\nabla'$ and $\nabla''$ are so.

The uniqueness of a Hermitian--Einstein connection on $E_G$
follows from the uniqueness of the Hermitian--Einstein connections
on the vector bundle $\text{ad}(E_G)$ and the
principal $G/[G\, ,G]$--bundle $E_G/[G\, ,G]$.
This completes the proof of the theorem.
\end{proof}

\subsection{A Bogomolov type inequality}\label{sec.b}

As before, $M$ is a compact connected special flat affine manifold,
$g$ is a Gauduchon metric on $M$, and $\nu$ is a nonzero covariant
constant volume form on $M$.
Let $d$ be the dimension of $M$. The $(1\, ,1)$--form
given by $g$ will be denoted by $\omega_g$. We recall that
the Gauduchon condition says that
$$
\partial\overline{\partial}(\omega^{d-1}_g)\, =\,0
$$
(see \cite[p. 109]{Lo}).

The Gauduchon metric $g$ is called
\textit{astheno--K\"ahler} if
\begin{equation}\label{aK}
\partial\overline{\partial}(\omega^{d-2}_g)\, =\,0
\end{equation}
(see \cite[p. 246]{JY}).

We note that if $d\,=\, 2$, then $g$ is astheno--K\"ahler.
If $g$ is K\"ahler, then $g$ is also astheno--K\"ahler.

We assume that the Gauduchon metric $g$ is astheno--K\"ahler.

Let $E$ be a flat vector bundle on $M$. Take a Hermitian structure
$h$ on $E$. Using \eqref{aK} it follows that
$$
\int_M \frac{c_1(E,h)^2\omega^{d-2}_g}{\nu}\,\in\,{\mathbb R}
\, ~\text{~and~}\, ~ \int_M \frac{c_2(E,h)\omega^{d-2}_g}{\nu}
\,\in\,{\mathbb R}
$$
are independent of the choice of $h$.

\begin{lemma}\label{lem6}
Let $V\, \longrightarrow\, M$ be a semistable flat vector bundle
of rank $r$. Then
$$
\int_M \frac{c_2({\mathcal E}nd(V))\omega^{d-2}_g}{\nu}\, =\,
\int_M \frac{(2r\cdot c_2(V) -
(r-1)c_1(V)^2)\omega^{d-2}_g}{\nu}\, \geq\, 0\, .
$$
\end{lemma}

\begin{proof}
First assume that $V$ is polystable. Therefore, $V$ admits a
Hermitian--Einstein connection (see \cite[p. 102, Theorem 1]{Lo} for
complex case and
\cite[p. 129, Corollary 33]{Lo} for real case). Let $h$ be
a Hermitian--Einstein metric on $V$. Then the $d$--form
$$
\frac{(2r\cdot c_2(V,h) -
(r-1)c_1(V,h)^2)\omega^{d-2}_g}{\nu}
$$
on $M$ is pointwise nonnegative (see
\cite{Lu} and \cite[p. 107]{LYZ} for the computation). Therefore,
the lemma is proved for polystable vector bundles.

If $V$ is semistable, then there is a filtration of flat subbundles
$$
0\,=\, V_0\, \subset\, V_1\, \subset\, \cdots
\, \subset\, V_{n-1} \, \subset\, V_\ell\, =\, V
$$
such that $V_i/V_{i-1}$ is polystable for all $i\,\in\, [1\, ,
n]$, and also $\mu_g(V_i/V_{i-1})\,=\, \mu_g(V)$. Since
the inequality in the statement of the lemma holds for all
$V_i/V_{i-1}$, it also holds for $V$.
\end{proof}

Let $G$ be a connected Lie group such that it is either
a complex reductive linear algebraic group
or it is the fixed point locus of a split real form on a
complex reductive linear algebraic group

\begin{proposition}\label{prop6}
Let $E_G\,\longrightarrow
\, M$ be a flat semistable principal $G$--bundle. Then
$$
\int_M \frac{c_2({\rm ad}(E_G))\omega^{d-2}_g}{\nu}\, \geq\, 0\, .
$$
\end{proposition}

\begin{proof}
The vector bundle ${\rm ad}(E_G)$ is semistable because $E_G$
is semistable (see Corollary \ref{cor2}). Since ${\rm ad}(E_G)\,
=\, {\rm ad}(E_G)^*$, we have
$$
\int_M \frac{c_1({\rm ad}(E_G))^2\omega^{d-2}_g}{\nu}\,=\,0\, .
$$
Hence the proof is completed by Lemma \ref{lem6}.
\end{proof}


\end{document}